\documentclass[reqno,11pt]{amsart}
\usepackage{graphicx}
\usepackage{amsmath,amsthm,amsfonts,amssymb,latexsym,epsfig}
\newtheorem{theorem}{Theorem}[section]
\newtheorem{lemma}[theorem]{Lemma}

\theoremstyle{definition}

\usepackage{hyperref}
\theoremstyle{remark}

\usepackage{orcidlink}
\usepackage[a4paper, left=1in, right=1in, top=1in, bottom=1in]{geometry}

\numberwithin{equation}{section}

\begin{document}

\title[short text for running head]{``2-color partitions modulo powers of $5$''}
\title{On 2-color partitions where one of the color is multiples of $5^k$ }

\author[]{Shivashankar C.$^{1}$\orcidlink{0000-0001-7503-4414}}
\address{$^{1}$Department of Mathematics, Vidyavardhaka College of Engineering, Gokulam III Stage, Mysuru-570002, Karnataka, India.}
\email{shankars224@gmail.com}
\thanks{}

\author[]{HemanthKumar B.$^{2}$\orcidlink{0000-0001-7904-293X}}

\address{$^2$Department of Mathematics, RV College of Engineering, RV Vidyanikethan Post, Mysore Road, Bengaluru-560 059, Karnataka, India.}
\email{hemanthkumarb.30@gmail.com}

\thanks{}

\author[]{D. S. Gireesh$^{3}$\orcidlink{0000-0002-2804-6479}}
\address{$^{3}$Department of Mathematics, BMS College of Engineering, P.O. Box No.: 1908, Bull Temple Road,
Bengaluru-560 019, Karnataka, India.}
\email{gireeshdap@gmail.com}
\thanks{}

\date{}
          
\begin{abstract}
In this work, we investigate the arithmetic properties of $p_{1,5^k}(n)$, which counts 2-color partitions of $n$ where one of the colors appears only in parts that are multiples of $5^k$. By constructing generating functions for $p_{1,5^k}(n)$ across specific arithmetic progressions, we establish a set of Ramanujan-type infinite family of congruences modulo powers of $5$.
\end{abstract}

\medskip

\maketitle

\section{Introduction}
\label{intro}
A partition of a positive integer $n$ is a non-increasing sequence of positive integers that sums is $n$. Let $p(n)$ denote the number of partitions of $n$, with the generating function given by
\[\sum\limits_{n\geq0}p(n)q^n=\frac1{f_1}.\]
Here and throughout the paper, we set
\[f_r:=(q^r;q^r)_\infty=\prod\limits_{m=1}^{\infty}(1-q^{rm}).\]

Ramanujan \cite{Ram2} conjectured, and Watson \cite{GNW} proved that
\begin{equation}\label{pc1}
    p(5^kn+\delta_k)\equiv 0\pmod{5^k},
\end{equation}
where $k\geq1$ and $\delta_k$ is the reciprocal modulo $5^k$ of $24$.

Hirschhorn and Hunt \cite{HH} proved \eqref{pc1} by establishing generating functions
\begin{equation}\label{H1}
\sum_{n\geq 0}p\left(5^{2k-1}n+\delta_{2k-1}\right)q^n=\sum_{j\geq 1}x_{2k-1,j} \, q^{j-1}\frac{f_5^{6j-1}}{f_1^{6j}}
\end{equation}
and
\begin{equation}\label{H2}
\sum_{n\geq 0}p\left(5^{2k}n+\delta_{2k}\right)q^n=\sum_{j\geq 1}x_{2k,j} \, q^{j-1}\frac{f_5^{6j}}{f_1^{6j+1}},
\end{equation}
where the coefficient vectors $\mathbf{x}_k= (x_{k,1}, x_{k,2},\dots)$ are given by
$$\mathbf{x}_1=\left(x_{1,1},x_{1,2},\dots\right)=\left(5,0,0,\dots\right),$$ and for $k\geq1$,
\[
   x_{k+1, i}=\begin{cases}
  \displaystyle \sum_{j\geq1} x_{k, j} \,m_{6j,j+i}, \,\,\text{if}\,\,k\,\,\text{is odd},\\
   \displaystyle \sum_{j\geq1}x_{k, j} \,m_{6j+1,j+i}, \,\,\text{if}\, \,k \,\,\text{is even},
    \end{cases}
\]
where the first five rows of $M=\left(m_{i,j}\right)_{i,j\geq 1}$ are
\[\begin{bmatrix}
5 & 0 & 0 & 0& 0& 0&\dots\\
2\times 5 & 5^3 & 0 &0 &0 & 0& \dots\\
9 & 3\times5^3 & 5^5 &0 &0 & 0& \dots\\
4 & 22\times5^2 & 4\times5^5 &5^7 &0 & 0& \dots\\
1 & 4\times5^3 & 8\times5^5 &5^8 &5^9 & 0& \dots
\end{bmatrix}\]
and for $i\geq6$, $m_{i,1}=0$, and for $j\geq2$,
\[m_{i,j}=25m_{i-1,j-1}+25m_{i-2,j-1}+15m_{i-3,j-1}+5m_{i-4,j-1}+m_{i-5,j-1}.\]

Let $p_{1,\ell}(n)$ be the number of 2-color partitions of $n$ where one of the colors appears only in parts that are multiples of $\ell$; its generating function is given by
\begin{equation}\label{bl}
\sum\limits_{n\geq0}p_{1,\ell}(n)q^n=\frac{1}{f_1f_{\ell}}.
\end{equation}

Ahmed, Baruah, and Dastidar \cite {ABD} found several new congruences modulo 5:
\begin{equation}\label{AB}
p_{1,\ell}\left(25n+t\right)\equiv 0 \pmod{5},
\end{equation}
where $\ell\in\{0, 1, 2, 3, 4, 5, 10, 15, 20\}$ and $\ell+t=24$. They also conjectured that the congruence also holds for $\ell= 7, 8,$ \text{and} $17$. This conjecture was confirmed by Chern \cite{SC}, moreover he proved that
\[p_{1,4}\left(49n+t\right)\equiv 0 \pmod{7},\]
where $t\in\{11, 25, 32, 39\}$.

In \cite{LW}, Wang derived congruences for $p_{1,5}(n)$ analogous to \eqref{pc1}. He proved that
\begin{align}
&p_{1,5}\left(5^{\beta+1}n+\frac{3\times5^{\beta+1}+1}{4}\right)\equiv0\pmod{5^{\beta+1}},\label{W1}\\&
p_{1,5}\left(5^{\beta+2}n+\frac{11\times5^{\beta+1}+1}{4}\right)\equiv0\pmod{5^{\beta+2}},\label{W2}\\&
p_{1,5}\left(5^{\beta+2}n+\frac{19\times5^{\beta+1}+1}{4}\right)\equiv0\pmod{5^{\beta+2}},\label{W2}
\end{align}
for each $n,\beta\geq 0$. Note that, \eqref{W1} is stronger result than \eqref{AB} for $\ell=5$.

Ranganath \cite{RD} extended these results to $p_{1,25}(n)$. For each $n,\beta \geq0$, he proved that
\begin{equation}\label{RD}
p_{1,25}\left(5^{2\beta+1}n+\frac{7\times5^{2\beta+1}+13}{12}\right)\equiv0\pmod{5^{\beta+1}}
\end{equation}
and
\begin{equation}\label{RD}
p_{1,25}\left(5^{2\beta+2}n+\frac{11\times5^{2\beta+2}+13}{12}\right)\equiv0\pmod{5^{\beta+2}}.
\end{equation}

We establish congruences for $p_{1,5^k}$ for all positive integers $k$. The congruences previously derived by Wang and Ranganath can be seen as specific instances of the results presented in this work.

The main results are as follow:
\begin{theorem}\label{th1}
For each $n,\beta \geq0$, and $k\geq1$, we have
\begin{equation}\label{c1}
    p_{1,5^{2k-1}}\left(5^{2k+\beta-1}n+\frac{18\cdot5^{2k+\beta-1}+5^{2k-1}+1}{24}\right)\equiv 0\pmod{5^{2k+\beta-1}},
\end{equation}
\begin{equation}\label{c2}
    p_{1,5^{2k}}\left(5^{2k+2\beta-1}n+\frac{14\cdot5^{2k+2\beta-1}+5^{2k}+1}{24}\right)\equiv 0\pmod{5^{2k+\beta-1}},
\end{equation}
\begin{equation}\label{c3}
    p_{1,5^{2k}}\left(5^{2k+2\beta}n+\frac{22\cdot5^{2k+2\beta}+5^{2k}+1}{24}\right)\equiv 0\pmod{5^{2k+\beta}},
\end{equation}
\begin{equation}\label{c4}
    p_{1,5^{2k-1}}\left(5^{2k+\beta}n+\frac{(24r+18)\cdot5^{2k+\beta-1}+5^{2k-1}+1}{24}\right)\equiv 0\pmod{5^{2k+\beta}},
\end{equation}
and
\begin{equation}\label{c6}
    p_{1,5^{2k}}\left(5^{2k+2\beta+2}n+\frac{22\cdot5^{2k+2\beta+2}+5^{2k}+1}{24}\right)\equiv 0\pmod{5^{2k+\beta+1}},
\end{equation}
where $r\in\{2,3,4\}$.
\end{theorem}

\section{Preliminaries}
In this section, we state some lemmas which play a vital role in proving our results.
Let $H$ be the “huffing” operator modulo $5$, that is,
\[H\left(\sum{p_nq^n}\right)=\sum{p_{1,5n}q^{5n}}.\]
\begin{lemma}[\cite{HH}]\label{LGu}
 If $G= \frac{f_5^6}{q^4 f_1 f_{25}^5}$ and $u = \frac{f_5^6}{q^5 f_{25}^6}$, then
\begin{equation}\label{Gu}
    H \left( G^i \right) = \sum_{j=1}^{i} m_{i,j} u^{i-j}.
\end{equation}
\end{lemma}

\begin{lemma}\label{Hi1}
For all $i\geq1$, we have

\begin{equation}\label{H6i+1}
H \left( q^{i+1} \frac{f_{5}^{6i}}{f_1^{6i+1}} \right)= \sum_{j=1}^{5i+1} m_{6i+1,i+j} q^{5j} \frac{f_{25}^{6j-1}}{f_5^{6j}},
\end{equation}
\begin{equation}\label{H6i+2}
H \left( q^{i+2} \frac{f_{5}^{6i}}{f_1^{6i+2}} \right)= \sum_{j=1}^{5i+2} m_{6i+2	,i+j} \, q^{5j} \frac{f_{25}^{6j-2}}{f_5^{6j}},
\end{equation}
and
\begin{equation}\label{H6i}
H \left( q^{i} \frac{f_{5}^{6i-1}}{f_1^{6i}} \right)= \sum_{j=1}^{5i} m_{6i, i+j} \, q^{5j} \frac{f_{25}^{6j}}{f_5^{6j+1}}.
\end{equation}
\end{lemma}
\begin{proof}
We can rewrite \eqref{Gu} as
\begin{equation}\label{Hi}
H \left( q^i \frac{f_{25}^i}{f_1^i} \right) = \sum_{j=1}^{i} m_{i,j} \, q^{5j} \frac{f_{25}^{6j}}{f_5^{6j}}.
\end{equation}
From \eqref{Hi} and the fact that $m_{6i+1, j}=0$ for $1\leq j < i,$ we have
\begin{align*}
H\left(\left(q\frac{f_{25}}{f_1}\right)^{6i+1}\right)&=\sum_{j=i+1}^{6i+1} m_{6i+1,j} \,q^{5j} \frac{f_{25}^{6j}}{f_{5}^{6j}}.\\
&=\sum_{j=1}^{5i+1} m_{6i+1,i+j} \,q^{5i+5j} \frac{f_{25}^{6i+6j}}{f_{5}^{6i+6j}},
\end{align*}
which yields \eqref{H6i+1}. Similarly, we can prove \eqref{H6i+2} and \eqref{H6i}.
\end{proof}

\section{Generating functions}
In this section, we establish generating functions for $p_{1,5^{\ell}}(n)$ within specific arithmetic progressions.
\begin{theorem}\label{T1}
For each $\beta \geq0$ and $k\geq1$, we have
\begin{equation}\label{G1}
    \sum_{n\geq 0}p_{1,5^{2k-1}}\left(5^{2k+\beta-1}n+\frac{18\cdot5^{2k+\beta-1}+5^{2k-1}+1}{24}\right)q^n=\sum_{i\geq 1}y^{(2k-1)}_{\beta+1,i} \, q^{i-1}\frac{f_5^{6i-1}}{f_1^{6i+1}}
\end{equation}
where the coefficient vectors are defined as follows:
\begin{equation*}
        y_{1,j}^{(2k-1)}=x_{2k-1,j}
\end{equation*}
and
\begin{equation*}
       y^{(2k-1)}_{\beta+1,j}= \sum_{i\geq1}y^{(2k-1)}_{\beta,i} m_{6i+1,j+i}
\end{equation*}
for all $\beta, j\geq 1$.
\end{theorem}
\begin{proof}
From \eqref{bl}, we have
\begin{equation}\label{p1}
    \sum\limits_{n\geq0}p_{1,5^{2k-1}}(n)q^n=\dfrac{1}{f_{5^{2k-1}}}\sum\limits_{n\geq0}p(n)q^n.
\end{equation}
Extracting the terms involving $q^{5^{2k-1}n+\delta_{2k-1}}$ on both sides of \eqref{p1} and dividing throughout by $q^{\delta_{2k-1}}$, we obtain  
\begin{equation*}
    \sum\limits_{n\geq0}p_{1,5^{2k-1}}(5^{2k-1}n+\delta_{2k-1})q^{5^{2k-1}n}=\dfrac{1}{f_{5^{2k-1}}}\sum\limits_{n\geq0}p(5^{2k-1}n+\delta_{2k-1})q^{5^{2k-1}n}.
\end{equation*}
If we replace $q^{5^{2k-1}}$ by $q$ and  use \eqref{H1}, we get
\begin{equation*}
    \sum\limits_{n\geq0}p_{1,5^{2k-1}}(5^{2k-1}n+\delta_{2k-1})q^n=\sum_{j\geq 1}x_{2k-1,j} \, q^{j-1}\frac{f_5^{6j-1}}{f_1^{6j+1}},
\end{equation*}
which is the case $\beta=0$ of \eqref{G1}.

We now assume that \eqref{G1} is true for some integer $\beta\geq 0$. Applying the operator $H$ to both sides, by \eqref{H6i+1}, we have

\begin{equation*}
\begin{split}
    \sum_{n\geq 0}p_{1,5^{2k-1}}\left(5^{2k+\beta}n+\frac{18\cdot5^{2k+\beta}+5^{2k-1}+1}{24}\right)q^{5n+5}&=\sum_{i\geq 1}y^{(2k-1)}_{\beta+1,i} \, H\left(q^{i+1}\frac{f_5^{6i}}{f_1^{6i+1}}\right)\times \frac{1}{f_5}
    \\&=\sum_{i\geq 1}y^{(2k-1)}_{\beta+1,i}\sum_{j=1}^{5i+1} m_{6i+1,i+j}\, q^{5j} \frac{f_{25}^{6j-1}}{f_5^{6j+1}}\\&
    =\sum_{j\geq 1}\left(\sum_{i\geq1}y^{(2k-1)}_{\beta+1,i} \, m_{6i+1,i+j}\right) q^{5j} \frac{f_{25}^{6j-1}}{f_5^{6j+1}}\\&
    =\sum_{j\geq 1}y^{(2k-1)}_{\beta+2,j} \,q^{5j} \frac{f_{25}^{6j-1}}{f_5^{6j+1}}.
    \end{split}
\end{equation*}
That is,
\begin{equation*}
    \sum_{n\geq 0}p_{1,5^{2k-1}}\left(5^{2k+\beta}n+\frac{18\cdot5^{2k+\beta}+5^{2k-1}+1}{24}\right)q^{n}=\sum_{i\geq 1}y^{(2k-1)}_{\beta+2,i} \, q^{i-1} \frac{f_{5}^{6i-1}}{f_1^{6i+1}}.
\end{equation*}
So we obtain \eqref{G1} with $\beta$ replaced by $\beta+1$.

\end{proof}

\begin{theorem} \label{T2}
For each $\beta \geq0$, and $k\geq1$, we have
\begin{equation}\label{G3}
    \sum_{n\geq 0}p_{1,5^{2k}}\left(5^{2k+2\beta-1}n+\frac{14\cdot5^{2k+2\beta-1}+5^{2k}+1}{24}\right)q^n=\sum_{i\geq 1}y^{(2k)}_{2\beta+1,i}\,q^{i-1}\frac{f_5^{6i-2}}{f_1^{6i}},
\end{equation}
and
\begin{equation}\label{G4}
    \sum_{n\geq 0}p_{1,5^{2k}}\left(5^{2k+2\beta}n+\frac{22\cdot5^{2k+2\beta}+5^{2k}+1}{24}\right)q^{n}=\sum_{j\geq 1}y^{(2k)}_{2\beta+2,j}q^{j-1} \frac{f_{5}^{6j}}{f_1^{6j+2}},
\end{equation}
where coefficient vectors are defined as follow:
\begin{equation*}
        y_{1,j}^{(2k)}=x_{2k-1,j}
\end{equation*}
and
\begin{equation*}
	y^{(2k)}_{\beta+1,j}= 
	\begin{cases}
		\displaystyle \sum_{i\geq1}y^{(2k)}_{\beta,i} m_{6i+2,j+i} \,\,\   & \text{if  $\beta$ is even},\\
		\displaystyle \sum_{i\geq1}y^{(2k)}_{\beta,i}m_{6i,j+i} \,\,\   & \text{if  $\beta$ is odd},
	\end{cases}
\end{equation*}
for all $\beta, j\geq 1$.
\end{theorem}
\begin{proof}
    Equation \eqref{bl} can be written as
\begin{equation*}
    \sum\limits_{n\geq0}p_{1,5^{2k}}(n)q^n=\dfrac{1}{f_{5^{2k}}}\sum\limits_{n\geq0}p(n)q^n.
\end{equation*}
Extracting the terms involving $q^{5^{2k-1}n+\delta_{2k-1}}$ on both sides of the above equation and dividing throughout by $q^{\delta_{2k-1}}$, we obtain  
\begin{equation}\label{p6}
    \sum\limits_{n\geq0}p_{1,5^{2k}}(5^{2k-1}n+\delta_{2k-1})q^{5^{2k-1}n}=\dfrac{1}{f_{5^{2k}}}\sum\limits_{n\geq0}p(5^{2k-1}n+\delta_{2k-1})q^{5^{2k-1}n}.
\end{equation}
If we replace $q^{5^{2k-1}}$ by $q$ and use \eqref{H2}, we get
\begin{equation}\label{p7}
    \sum\limits_{n\geq0}p_{1,5^{2k}}(5^{2k-1}n+\delta_{2k-1})q^n=\sum_{j\geq 1}x_{2k-1,j} \, q^{j-1}\frac{f_5^{6j-2}}{f_1^{6j}},
\end{equation}
which is the case $\beta=0$ of \eqref{G3}.

We now assume that \eqref{G3} is true for some integer $\beta\geq0$. Applying the operator $H$ to both sides, by \eqref{H6i+2}, we have
\begin{equation*}
    \begin{split}
      \sum_{n\geq 0}p_{1,5^{2k}}\left(5^{2k+2\beta}n+\frac{22\cdot5^{2k+2\beta}+5^{2k}+1}{24}\right)q^{5n+5}&=\sum_{i\geq 1}y^{(2k)}_{2\beta+1,i} \, H\left(q^{i}\frac{f_5^{6i-1}}{f_1^{6i}}\right)\times\frac{1}{f_5}\\&=\sum_{i\geq 1}y^{(2k)}_{2\beta+1,i}\sum_{j=1}^{5i} m_{6i, i+j} \,q^{5j} \frac{f_{25}^{6j}}{f_5^{6j+2}}\\&
        =\sum_{j\geq 1}\left(\sum_{i\geq1}y^{(2k)}_{2\beta+1,i} \, m_{6i, i+j} \right)q^{5j} \frac{f_{25}^{6j}}{f_5^{6j+2}}\\&
        =\sum_{j\geq 1}y^{(2k)}_{2\beta+2,j}\, q^{5j} \frac{f_{25}^{6j}}{f_5^{6j+2}},
    \end{split}
\end{equation*}
That is
\begin{equation*}
    \sum_{n\geq 0}p_{1,5^{2k}}\left(5^{2k+2\beta}n+\frac{22\cdot5^{2k+2\beta}+5^{2k}+1}{24}\right)q^{n}=\sum_{j\geq 1}y^{(2k)}_{2\beta+2,j}q^{j-1} \frac{f_{5}^{6j}}{f_1^{6j+2}}.
\end{equation*}
Hence, if \eqref{G3} is true for some integer $\beta \geq 0$, then \eqref{G4} is true for $\beta$.

Suppose that \eqref{G4} is true for some integer $\beta \geq 0$. Applying the operating $H$ to both sides, by \eqref{H6i}, we obtain
\begin{equation*}
\begin{split}
       \sum_{n\geq 0}p_{1,5^{2k}}\left(5^{2k+2\beta+1}n+\frac{14\cdot5^{2k+2\beta+1}+5^{2k}+1}{24}\right)q^{5n+5}&=\sum_{i\geq 1}y^{(2k)}_{2\beta+2,i} \sum_{j=1}^{5i} m_{6i+2,i+j} q^{5j} \, \frac{f_{25}^{6j-2}}{f_5^{6j}}\\&
       =\sum_{j\geq 1}\left(\sum_{i\geq1}y^{(2k)}_{2\beta+2,i} \, m_{6i+2,i+j}\right) q^{5j} \frac{f_{25}^{6j-2}}{f_5^{6j}}\\&
       =\sum_{j\geq 1}y^{(2k)}_{2\beta+3,j} \, q^{5j} \frac{f_{25}^{6j-2}}{f_5^{6j}},
\end{split}
\end{equation*}
which yields
\begin{equation*}
    \sum_{n\geq 0}p_{1,5^{2k}}\left(5^{2k+2\beta+1}n+\frac{14\cdot5^{2k+2\beta+1}+5^{2k}+1}{24}\right)q^{n}=\sum_{j\geq 1}y^{(2k)}_{2\beta+3,j} \, q^{j-1} \frac{f_{5}^{6j-2}}{f_1^{6j}}.
\end{equation*}
This is \eqref{G3} with $\beta$ replaced by $\beta+1$. This completes the proof.
\end{proof}


\section{Proof of Congruences}

For a positive integer $n$, let $\pi(n)$ be the highest power of $5$ that divides $n$, and define $\pi(0)=+\infty$.
\begin{lemma}[\cite{HH}, Lemma 4.1]
For each $i,j\geq1 $, we have
\begin{equation}\label{pi3}
        \pi\left(m_{i,j}\right)\geq\frac{5j-i-1}{2}.
    \end{equation}
\end{lemma}
\begin{lemma}[\cite{HH}, Lemma 4.3]
    For each $k,j\geq1 $, we have
    \begin{equation}\label{pi1}
   \pi\left(x_{2k-1,j}\right)\geq 2k-1+\left[\frac{5j-5}{2}\right]
     \end{equation} 
    \end{lemma}
\begin{lemma}
    For each $j,k\geq1$ and $\beta\geq 0$, we have
\begin{equation}\label{pi4}
   \pi\left(y^{(2k-1)}_{\beta+1,j}\right)\geq 2k+\beta-1+\left[\frac{5j-5}{2}\right].
\end{equation}
\end{lemma}
\begin{proof}
In view of Theorem \ref{T1}, we have
\begin{equation*}
        y_{1,j}^{(2k-1)}=x_{2k-1,j}.
\end{equation*}
From \eqref{pi1}, we can see that the inequality \eqref{pi4} holds for $\beta=0$.

We now assume that \eqref{pi4} is true for some $\beta\geq 0$, then
\begin{align*}
   \pi\left(y^{(2k-1)}_{\beta+2,j}\right)&\geq \min_{i\geq1}\left\{\pi\left(y^{(2k-1)}_{\beta+1,i}\right)+\pi\left(m_{6i+1,i+j}\right)\right\}\\&
   \geq \min_{i\geq1}\left\{2k+\beta-1+\left[\frac{5i-5}{2}\right]+\left[\frac{5j-i-2}{2}\right]\right\}\\&
   \geq 2k+\beta+\left[\frac{5j-5}{2}\right],
   \end{align*}

which is \eqref{pi4} with $\beta+1$ for $\beta$. This completes the proof.
\end{proof}
\begin{lemma}
    For each $j,k\geq1$ and $\beta\geq 0$, we have
\begin{align}\label{pi6}
   \pi\left(y^{(2k)}_{2\beta+1,j}\right)\geq 2k+\beta-1+\left[\frac{5j-5}{2}\right]
   \end{align}
   and
\begin{align}\label{pi61}
\pi\left(y^{(2k)}_{2\beta+2,j}\right)\geq 2k+\beta+\left[\frac{5j-5}{2}\right].
\end{align}
\end{lemma}
\begin{proof}
In view of Theorem \ref{T2}, we have
\begin{equation*}
        y_{1,j}^{(2k)}=x_{2k-1,j}.
\end{equation*}
From \eqref{pi1}, we can see that \eqref{pi6} holds for $\beta=0$.
Now assume that \eqref{pi6} is true for some $\beta\geq 0$, then
\begin{align*}
   \pi\left(y^{(2k)}_{2\beta+2,j}\right)&\geq \min_{i\geq1}\left\{\pi\left(y^{(2k)}_{2\beta+1,i}\right)+\pi\left(m_{6i,i+j}\right)\right\}\\&
   \geq \min_{i\geq1}\left\{2k+\beta-1+\left[\frac{5i-4}{2}\right]+\left[\frac{5j-i-1}{2}\right]\right\}\\&
   \geq 2k+\beta-1+\left[\frac{5j-2}{2}\right]\\&
   = 2k+\beta+\left[\frac{5j-4}{2}\right]\\&
   \geq 2k+\beta+\left[\frac{5j-5}{2}\right],
   \end{align*}
which is \eqref{pi61}.

Now suppose \eqref{pi61} holds for some $\beta\geq0$. Then,
\begin{align*}
   \pi\left(y^{(2k)}_{2\beta+3,j}\right)&\geq \min_{i\geq1}\left\{\pi\left(y^{(2k)}_{2\beta+2,i}\right)+\pi\left(m_{6i+2,i+j}\right)\right\}\\&
   \geq \min_{i\geq1}\left\{2k+\beta+\left[\frac{5i-5}{2}\right]+\left[\frac{5j-i-3}{2}\right]\right\}\\&
   \geq 2k+\beta+\left[\frac{5j-4}{2}\right]\\&
   \geq2k+\beta+\left[\frac{5j-5}{2}\right],
\end{align*}
which is \eqref{pi6} with $\beta+1$ for $\beta$. This completes the proof.
\end{proof}
\noindent\textbf{Proof of Theorem \ref{th1}}
Congruence \eqref{c1} follows from \eqref{pi4} together with \eqref{G1}, \eqref{c2} follows from \eqref{pi6} together with \eqref{G3}, and \eqref{c3} follows from \eqref{pi61} and \eqref{G4}.

Let $p_{-2}(n)$ be defined by 
\begin{equation}
\sum\limits_{n=0}^{\infty}p_{-2}(n)q^n=\frac{1}{f_1^2}.
\end{equation}
It has been shown by Ramanathan \cite{KGR} that for $n\ge0$ and $s\in\{2, 3, 4\}$,
\begin{equation}\label{d1}
p_{-2}(5n+s)\equiv 0\pmod{5}.
\end{equation}

In view of \eqref{pi4}, we can express \eqref{G1} as
\begin{equation*}
\sum_{n\geq 0}p_{1,5^{2k-1}}\left(5^{2k+\beta-1}n+\frac{18\cdot5^{2k+\beta-1}+5^{2k-1}+1}{24}\right)q^n\equiv y^{(2k-1)}_{\beta+1,1}\frac{f_5^{5}}{f_1^{7}}\pmod{5^{2k+\beta+1}}.
\end{equation*}
Again from \eqref{pi4} and the binomial theorem, we have
\begin{align*}
\sum_{n\geq 0}p_{1,5^{2k-1}}\left(5^{2k+\beta-1}n+\frac{18\cdot5^{2k+\beta-1}+5^{2k-1}+1}{24}\right)q^n\equiv y^{(2k-1)}_{\beta+1,1}\frac{f_5^{4}}{f_1^{2}}\pmod{5^{2k+\beta}},
\end{align*}
which implies that
\begin{equation}\label{c15}
\sum_{n\geq 0}p_{1,5^{2k-1}}\left(5^{2k+\beta-1}n+\frac{18\cdot5^{2k+\beta-1}+5^{2k-1}+1}{24}\right)q^n\equiv y^{(2k-1)}_{\beta+1,1} f_5^{4}\sum\limits_{n=0}^{\infty}p_{-2}(n)q^n\pmod{5^{2k+\beta}}.
\end{equation}
Congrunence \eqref{c4}, follows from \eqref{c15} and \eqref{d1}.


By \eqref{pi61}, we can express \eqref{G4} as
\begin{equation*}
\sum_{n\geq 0}p_{1,5^{2k}}\left(5^{2k+2\beta}n+\frac{22\cdot5^{2k+2\beta}+5^{2k}+1}{24}\right)q^{n}\equiv y^{(2k)}_{2\beta+2,1} \frac{f_{5}^{6}}{f_1^{8}}\pmod{5^{2k+\beta}}.
\end{equation*}
Using \eqref{pi61} and the binomial theorem, we have
\begin{equation}\label{c18}
\sum_{n\geq 0}p_{1,5^{2k}}\left(5^{2k+2\beta}n+\frac{22\cdot5^{2k+2\beta}+5^{2k}+1}{24}\right)q^{n}\equiv y^{(2k)}_{2\beta+2,1} f_5^4f_1^2\pmod{5^{2k+\beta+1}}.
\end{equation}
From \cite{Ram2}, we have
\begin{equation}\label{c19}
f_1=f_{25}(R(q^5)-q-q^2R^{-1}(q^5)), \,\,\,\, \text{where} \,\,\,\, R(q)=\dfrac{f(-q^2,-q^3)}{f(-q,-q^4)}.
\end{equation}
Invoking \eqref{c19} in \eqref{c18}, we obatain
\begin{equation}\label{c20}
\begin{split}
&\sum_{n\geq 0}p_{1,5^{2k}}\left(5^{2k+2\beta}n+\frac{22\cdot5^{2k+2\beta}+5^{2k}+1}{24}\right)q^{n}\\&\equiv y^{(2k)}_{2\beta+2,1} f_5^4f_{25}^2\left(R(q^5)-q-q^2R^{-1}(q^5)\right)^2\pmod{5^{2k+\beta+1}}.
\end{split}
\end{equation}
Equating the coefficients of $q^{5n+2}$ from both sides of the above equation, dividing throughout by $q^2$ and then replacing $q^5$ by $q$,
\begin{equation*}
\sum_{n\geq 0}p_{1,5^{2k}}\left(5^{2k+2\beta+1}n+\frac{70\cdot5^{2k+2\beta}+5^{2k}+1}{24}\right)q^{n}\equiv -2y^{(2k)}_{2\beta+2,1} f_1^4f_{5}^2\pmod{5^{2k+\beta+1}}.
\end{equation*}
But
\begin{equation*}
    f_5^2f_1^4 \equiv \dfrac{f_5^3}{f_1}\pmod{5}.
\end{equation*}
Thus,
\begin{equation}\label{c21}
\sum_{n\geq 0}p_{1,5^{2k}}\left(5^{2k+2\beta+1}n+\frac{70\cdot5^{2k+2\beta}+5^{2k}+1}{24}\right)q^{n-4}\equiv y^{(2k)}_{2\beta+2,1} f_{5}^3\sum\limits_{n=0}^{\infty}p(n)q^{n-4}\pmod{5^{2k+\beta+1}}.
\end{equation}
Operating $H$ on both sides and then using \eqref{pi61} and \eqref{pc1}, we arrive at \eqref{c6}.

\end{document}